\def\marker{\>\hbox{${\vcenter{\vbox{\hrule height 0.4pt\hbox{\vrule width 0.4pt height 6pt \kern6pt\vrule width 0.4pt}\hrule height 0.4pt}}}$}\>}

\documentclass[12pt]{article}
\usepackage{indentfirst}
\usepackage{epsfig}
\usepackage{fullpage}
\usepackage{amsmath,amsfonts,amssymb,amsthm,graphics,amscd}
\usepackage{pgf,tikz}
\usepackage{graphicx}
\usepackage{amsmath,amsfonts,amssymb,amsthm,graphics,amscd,pst-plot,pstricks}
\usepackage{graphicx}
\usepackage{subcaption}

\newtheorem{Theorem}{Theorem}[section]
\newtheorem {Lemma}{Lemma}[section]
\newtheorem {Proposition}[Theorem]{Proposition}

\theoremstyle{definition}

\newtheorem{Conjecture}{Conjecture}

\newcommand{\CC}{\mathcal{C}}

\usepackage[
pdfauthor={},
pdftitle={Hilton-Zhao vertex-splitting conjecture for triangle-free graphs},
pdfstartview=XYZ,
bookmarks=true,
colorlinks=true,
linkcolor=blue,
urlcolor=blue,
citecolor=blue,
bookmarks=false,
linktocpage=true,
hyperindex=true
]{hyperref}

\begin{document} 
\begin{center}
{\Large \bf Hilton-Zhao vertex-splitting conjecture for triangle-free graphs}

\vspace{6mm}
{\large Xuli Qi$^{a}$, Xin Tian$^{a}$, Gregory Gutin$^{b}$}
\vspace{4mm}

\baselineskip=0.25in
{\it $^a$ Department of Mathematics, Hebei University of Science and Technology, China.}\\
{\it $^b$ Department of Computer Science, Royal Holloway University of London, UK.}\\

\vspace{3mm}
\baselineskip=0.25in
\footnotetext{{\it E-mail addresses:}  {\tt xuliqi@hebust.edu.cn}; {\tt 2025110016@stu.hebust.edu.cn}; {\tt g.gutin@rhul.ac.uk}}
\vspace{3mm}
\end{center}

\begin{abstract}
Every graph obtained from a regular Class~$1$ graph by a
vertex-splitting is overfull and hence is of Class~$2$; the central question is whether the resulting graph is edge-chromatic critical.
We consider a revised version of the Hilton-Zhao vertex-splitting conjecture for triangle-free base graphs, excluding the known cubic exception that arises from the Wagner graph.
Let $G$ be a connected $n$-vertex triangle-free $\Delta$-regular Class~$1$ graph, and let $G^{*}$ be obtained from $G$ by splitting one vertex into two adjacent vertices.
We first show that any non-extreme vertex-splitting produces a
$\Delta$-critical graph provided $\Delta>\frac{n-1}{3}$. Consequently, for non-extreme vertex-splittings, the revised version of the conjecture is valid across its entire conjectured degree range, and indeed even under a slightly weaker degree condition.
We next turn to arbitrary vertex-splittings, extreme ones included. It is shown that $G^*$ is $\Delta$-critical provided $n \ge 10$ and $\Delta>\frac{2n-3}{5}$. In addition, for $n>24$, the non-strict condition $\Delta \ge \frac{2n-3}{5}$ is already sufficient.
This extends the bipartite range $\Delta>\frac{2n}{5}$ to the boundary range that the bipartite argument does not cover.
In particular, our arguments give a complete treatment of the case of triangle-free $6$-regular base graphs in the conjectured range. The proofs bring together structural properties of full-deficiency pairs, homomorphism theorems for dense triangle-free graphs, and Hamilton decompositions built via graph detachment.

\end{abstract}
	
\section{Introduction}
Unless otherwise stated, we consider simple connected graphs in this paper. Nevertheless, in Lemmas $2.8$ and $4.4$, we consider multigraphs, which are defined as graphs allowing parallel edges and loops. For general notation and terminology, we follow the book \cite{SSTF2012} of Stiebitz et al.; the specific conventions for multigraphs are detailed later. Let $G=(V(G),E(G))$ be a graph, where $V(G)$ is the vertex set and $E(G)$ is the edge set of $G$. Two graphs $G$ and $H$ are {\bf isomorphic}, denoted by $G \cong H$, if there exists a bijection $f:V(G) \to V(H)$ such that $uv \in E(G)$ if and only if $f(u)f(v) \in E(H)$. Let $\mathbb{Z}$ denote the set of integers. For two integers $s$ and $t$, let $[s,t]:=\{i \in \mathbb{Z}:s \le i \le t\}$.  

For $u,v \in V(G)$, let $\operatorname{dist}_G(u,v)$ be the length of a shortest path between $u$ and $v$ in $G$. For a vertex $u \in V(G)$ and a subset $S \subseteq V(G)$, define $\operatorname{dist}_G(u,S):= \min_{s \in S} \operatorname{dist}_G(u,s)$. We also define $N_G(u):= \{v \in V(G):\operatorname{dist}_G(u,v)=1\}$, $d_G(u):= |N_G(u)|$, and $N_G[u]:= N_G(u) \cup \{u\}$. Here $d_G(u)$ is the degree of $u$, $N_G(u)$ is its open neighborhood, and $N_G[u]$ is its closed neighborhood. Denote by $\Delta(G)$ and $\delta(G)$ (or simply $\Delta$ and $\delta$) respectively the maximum degree and minimum degree of graph $G$. A graph $G$ is called $\Delta$-regular if all vertices of $G$ have the same degree $\Delta$, that is, $\Delta(G)=\delta(G)=\Delta$. 

A {\bf proper $k$-edge-coloring} of $G$ is a map $\varphi:E(G)\rightarrow\ [1,k]$ that assigns to every edge $e$ of $G$ a color $\varphi(e) \in [1,k]$ such that no two adjacent edges of $G$ receive the same color. Let $\CC^k(G)$ denote the collection of all proper $k$-edge-colorings of $G$. The {\bf chromatic index} $\chi'(G)$ is the least integer $k$ such that $\CC^k(G)\ne\emptyset$. By the celebrated results of Vizing \cite {Vizing1965-1} and Gupta \cite{Gupta1967}, every simple graph satisfies $\Delta(G)\le \chi'(G)\le\Delta(G)+1$. Following Fiorini and Wilson \cite {FW1977}, $G$ is of {\bf Class 1} if $\chi'(G)=\Delta(G)$; otherwise, it is of {\bf Class 2}. Determining whether an arbitrary graph belongs to Class 1 is computationally difficult; Holyer \cite{Hoyler1981} proved that this decision problem is NP-complete. We know that each color class in a proper edge-coloring is a matching and therefore contains at most $\lfloor |V(G)|/2 \rfloor$ edges. Consequently, $\chi'(G) \ge |E(G)|/\lfloor |V(G)|/2 \rfloor$. We say that $G$ is {\bf overfull} if $|E(G)|>\Delta(G) \lfloor |V(G)|/2 \rfloor$. Hence, an overfull simple graph has chromatic index $\Delta(G)+1$ and is therefore of Class 2. Notice also that an overfull graph necessarily has odd order. 
	
A graph $G$ is called {\bf edge-chromatic critical} if $\chi'(H)<\chi'(G)$ for every proper subgraph $H$ of $G$. In particular, if $G$ is edge-chromatic critical and $\chi'(G)=\Delta(G)+1$, then $G$ is called a {\bf $\Delta$-critical graph}, where $\Delta:=\Delta(G)$. In this paper, we are concerned with sufficient conditions for a Class 2 graph to be edge-chromatic critical. Let $G$ be a graph and $e \in E(G)$. We write $G-e$ for the graph obtained from $G$ by deleting the edge $e$. An edge $e$ is called a {\bf critical edge} if $\chi'(G-e)<\chi'(G)$. Thus, a Class 2 graph is edge-chromatic critical precisely when every edge is critical. A {\bf vertex-splitting} in a graph $G$ at a vertex $v$ forms a new graph $G^*$: replacing $v \in V(G)$ with two new adjacent vertices $v_1$ and $v_2$ and partitioning $N_G(v)$ into two nonempty subsets that serve as the set of neighbors of $v_1$ and $v_2$ from $V(G)$ in $G^*$, respectively. We say the new graph $G^\ast$ is obtained from $G$ by a vertex-splitting. We call this vertex-splitting {\bf extreme} if $\max\{d_{G^\ast}(v_{1}),d_{G^\ast}(v_{2})\}=\Delta(G)$. Otherwise, we call it \textbf{non-extreme}; that is, $d_{G^\ast}(v_{1})<\Delta(G)$ and $d_{G^\ast}(v_{2})<\Delta(G)$. Vertex-splitting was called the {\rmfamily ``M\"obius-type gluing technique''} in \cite{BV2020,Vietri2015}. It is easy to check that a regular Class 1 graph has even order, and that every graph obtained from a regular graph of even order by a vertex-splitting is overfull.  
	
Let $\Gamma$ (see Figure~\ref{fig:mobius_ladder}) denote the Wagner graph, i.e., eight-vertex Möbius ladder with $V(\Gamma)=\{x_0,\dots,x_7\}$ and $N_{\Gamma}(x_i)=\{x_{i-1},x_{i+1},x_{i+4}\}$, where the subscripts are taken modulo $8$. Splitting a vertex does not always produce an edge-chromatic critical graph. As noted in \cite{HZ1997,Yap1981}, the graph obtained from the Wagner graph \(\Gamma\) by subdividing one of its four opposite chords $\{x_0x_4,x_1x_5,x_2x_6,x_3x_7\}$ is not edge-chromatic critical. 
Therefore, excluding the known cubic exception mentioned above, we consider a revised version of the Hilton-Zhao conjecture \cite{HZ1997} as follows.

\begin{Conjecture}[Vertex-splitting conjecture]

Let $G$ be an $n$-vertex $\Delta$-regular Class~$1$
graph satisfying $\Delta>\frac{n}{3}$,
and let $G^{*}$ be obtained from $G$ by a vertex-splitting. If $G^{*}$ is not isomorphic to the graph obtained from the
Wagner graph $\Gamma$ by subdividing one of its four opposite chords, then $G^{*}$ is $\Delta$-critical.
\end{Conjecture}

\begin{figure}[htbp]
    \begin{center}
        \centering
        \scalebox{0.4}{\includegraphics{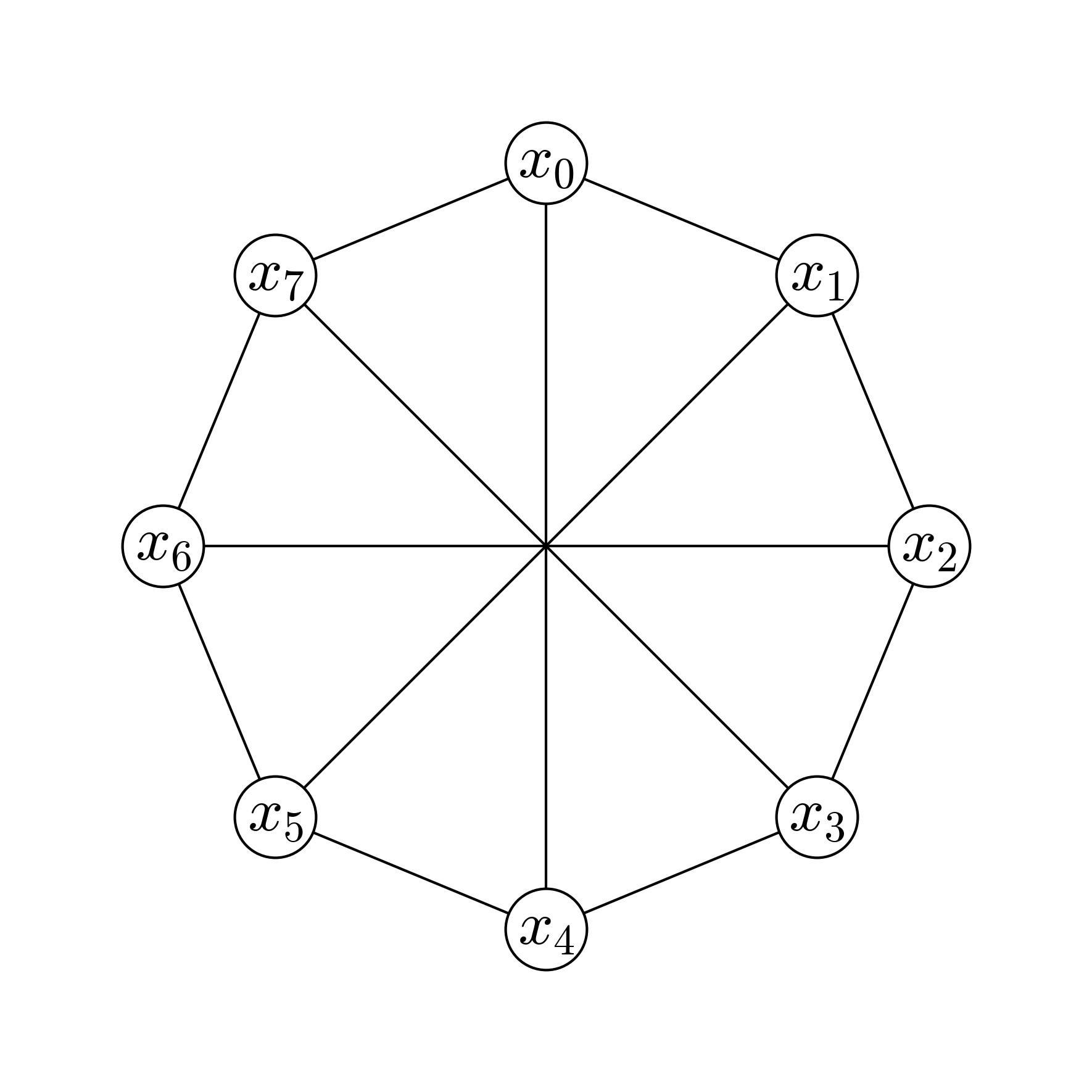}}
        \caption{Wagner graph $\Gamma$}
        \label{fig:mobius_ladder}
    \end{center}
\end{figure}

Since any graph obtained from a regular Class~$1$ graph by a vertex-splitting is overfull and hence is of Class~$2$, the central question is whether the resulting graph is edge-chromatic critical.
Some additional small-degree cases are covered by the classifications of small-order critical graphs in \cite{BBG2005, BS1998, CY1983}. Several partial results have been obtained. Hilton and Zhao \cite{HZ1997} proved the conjecture for graphs $G$ satisfying $\Delta(G)\ge\frac{\sqrt{7}-1}{2}n+2\approx0.82n$. In 2002, 
Song \cite{Song2002} verified the conjecture for non-extreme
vertex-splittings when diameter at most two,
in particular when $\Delta(G)\ge 0.5n$.
In 2022, Cao, Chen, and Shan \cite{CCS2022} verified the conjecture under the condition $\Delta(G)\ge0.75n$. Recently, Qi and Feng \cite{QF2026} proved the conjecture for $\Delta(G) \ge \frac{2n-2}{3}$.

In this paper, we consider a revised version of the Hilton-Zhao vertex-splitting conjecture for triangle-free base graphs. Let $G$ be an $n$-vertex triangle-free $\Delta$-regular Class~$1$ graph, and let $G^{*}$ be obtained from $G$ by a vertex-splitting. Our first result shows that, for non-extreme vertex-splittings, the revised version of the conjecture is valid across its entire conjectured degree range, and indeed even under a slightly weaker degree condition. We also give a complete treatment of triangle-free $6$-regular base graphs in the conjectured range. Finally, for arbitrary vertex-splittings, extreme ones included, we extend the bipartite range $\Delta>\frac{2n}{5}$ to the boundary range that the bipartite argument does not cover. The precise statements of our four main results are given below.

\begin{Theorem}\label{T1.1}
Let $G$ be an $n$-vertex  triangle-free $\Delta$-regular Class 1 graph with $\Delta(G)>\frac{n-1}{3}$ and $G^\ast$ be obtained from $G$ by splitting a vertex $v$ into two adjacent vertices $a$ and $b$. If $d_{G^\ast}(a)<\Delta(G)$ and $d_{G^\ast}(b)<\Delta(G)$, then $G^\ast$ is $\Delta$-critical.  
\end{Theorem}
	
\begin{Theorem}\label{T1.2}
Let $G$ be an $n$-vertex triangle-free 6-regular Class 1 graph. If $6>\frac{n}{3}$, then every graph obtained from $G$ by a vertex-splitting is 6-critical.
\end{Theorem}
	
\begin{Theorem}\label{T1.3}
Let $G$ be an $n$-vertex  triangle-free $\Delta$-regular Class 1 graph. If $\Delta(G) \ge \frac{2n-3}{5}$ with $n>24$, then every graph obtained from $G$ by a vertex-splitting is $\Delta$-critical.
\end{Theorem}

\begin{Theorem}\label{T1.4}
Let $G$ be an $n$-vertex triangle-free $\Delta$-regular Class 1 graph. If $\Delta(G)>\frac{2n-3}{5}$ with $n \ge 10$, then every graph obtained from $G$ by a vertex-splitting is $\Delta$-critical.
\end{Theorem}

The remainder of this paper is structured as follows. In Section 2, we introduce the terminology and notation and collect the preliminary results. In Section 3, we prove Theorem \ref{T1.1}, which concerns non-extreme vertex-splittings. Section 4 is dedicated to developing the auxiliary results for extreme vertex-splittings. In Section 5, we prove Theorem \ref{T1.2}, while Section 6 is devoted to the proofs of Theorems \ref{T1.3} and \ref{T1.4}.

\section{Definitions and preliminary results}
Throughout the paper, we write $\Delta:=\Delta(G)$ and $\delta:=\delta(G)$ for the maximum degree and minimum degree of $G$, respectively. An independent set in $G$ is a subset $S \subseteq V(G)$ such that no two vertices in $S$ are adjacent, meaning that for any $u,v \in S$, $uv \notin E(G)$; a maximum independent set is an independent set of the largest possible size in $G$, and the {\bf independence number} $\alpha(G)$ is the number of vertices in a maximum independent set of $G$, that is, $\alpha(G):=\max\{|S|:S \subseteq V(G) \text{ and } S \text{ is an independent set of } G\}$. A graph $G$ is {\bf bipartite} if its vertex set can be partitioned into two independent sets $X$ and $Y$. If every vertex of $X$ is adjacent to every vertex of $Y$, then $G$ is called a {\bf complete bipartite graph}. The complete bipartite graph with partite sets of sizes $m$ and $n$ is denoted by $K_{m,n}$. A complete bipartite graph $K_{m,n}$ is called {\bf balanced} if its two partite sets differ in size by at most one, that is, $|m-n|\le 1$. Thus, the balanced complete bipartite graph with order $n$ is $K_{\lfloor \frac{n}{2} \rfloor,\lceil \frac{n}{2} \rceil}$. A graph {\bf homomorphism} from $G$ to $H$ is a map $f:V(G) \to V(H)$ such that if $uv \in E(G)$, then $f(u)f(v) \in E(H)$. We write $G \to H$ if such a homomorphism exists. Let $G[R]$ denote the subgraph of $G$ induced by $R$, that is, the graph whose vertex set is $R$ and whose edge set consists of all edges of $G$ with both endvertices in $R$. 

A matching in a graph $G$ is a set of pairwise vertex-disjoint edges. A matching is perfect if every vertex of $G$ is incident with exactly one of its edges. If a graph $H$ is a subgraph of a graph $G$ and $H$ contains all the vertices of $G$, then $H$ is called a {\bf spanning subgraph} of $G$. A 1-factor of $G$ is a spanning 1-regular subgraph of $G$, or equivalently, a perfect matching. A {\bf 1-factorization} of $G$ is a partition of $E(G)$ into 1-factors. Thus, $F_1,\dots,F_k$ form a 1-factorization of $G$ if each $F_i$ with $i \in [1,k]$ is a perfect matching of $G$ and $E(G)=F_1\dot{\cup}\cdots\dot{\cup}F_k$, where $\dot{\cup}$ denotes the disjoint union. A Hamilton cycle of $G$ is a cycle containing every vertex of $G$; equivalently, it is a connected spanning 2-regular subgraph of $G$. A {\bf Hamilton decomposition} of $G$ is a partition of $E(G)$ into pairwise edge-disjoint Hamilton cycles. In other words, $H_1,\dots,H_r$ form a Hamilton decomposition of $G$ if each $H_i$ with $i \in [1,r]$ is a Hamilton cycle of $G$ and $E(G)=E(H_1)\dot{\cup}\cdots\dot{\cup}E(H_r)$.

When $u,v \in V(G)$ are adjacent, we call $(u,v)$ a {\bf full-deficiency pair} if $d_{G}(u)+d_{G}(v)=\Delta+2$. 

\begin{Lemma}[\cite{CCS2022}]\label{L2.1}
If $G$ is an $n$-vertex Class 2 graph with a full-deficiency pair $(a,b)$ such that $ab$ is a critical edge of $G$, then $G$ satisfies the following properties.
\begin{itemize}
\item [{$\mathrm{(1)}$}] For every $x \in (N_{G}(a) \cup N_{G}(b)) \setminus \{a,b\}$, $d_{G}(x) = \Delta$;
\item [{$\mathrm{(2)}$}] For every $x \in V(G) \setminus \{a,b\}$, if $\operatorname{dist}_{G}(x,\{a,b\})=2$, then $d_{G}(x) \ge \Delta-1$. Furthermore, if $d_{G}(a)<\Delta$ and $d_{G}(b)<\Delta$, then $d_{G}(x)=\Delta$. 
\end{itemize}
\end{Lemma}
	
\begin{Lemma}[\cite{HZ1997}]\label{L2.2}
Let $G$ be a regular Class 1 graph with $\delta \ge 2$, and let $G^*$ be obtained from $G$ by splitting a vertex $x$ into two adjacent vertices $a$ and $b$. If $e \in E(G^*)$ is incident with at least one of $a$ and $b$, then $G^*-e$ is of Class 1. Consequently, $e$ is a critical edge of $G^*$.
\end{Lemma} 

The following lemma is the $r = 3$ case of Theorem 1.1 of Andrásfai, Erdős, and Sós \cite{AES1974}, established in their Lemma $1.2$.

\begin{Lemma}[\cite{AES1974}]\label{L2.3}
If $G$ is an $n$-vertex triangle-free graph with $\delta>\frac{2n}{5}$, then $G$ is bipartite.  
\end{Lemma} 

It is well known that every bipartite graph is of Class 1 \cite{FW1977}.

\begin{Lemma}[\cite{HZ1997}]\label{L2.4}
Let $G$ be a $\Delta$-regular bipartite graph with $\Delta \ge 2$. Let $G^\ast$ be obtained from $G$ by a vertex-splitting. Then $G^{*}$ is $\Delta$-critical.
\end{Lemma}

The original lemma above establishes that every loopless multigraph obtained by vertex-splitting a regular bipartite multigraph is edge-chromatic critical, where parallel edges are allowed.
In particular, this result applies to regular bipartite simple graphs. Combining this result with Lemma \ref{L2.3}, we obtain the following consequence.
 
\begin{Proposition}\label{C2.1}
Let $G$ be an $n$-vertex triangle-free $\Delta$-regular Class 1 graph. If $\Delta>\frac{2n}{5}$, then every graph obtained from $G$ by a vertex-splitting is $\Delta$-critical.  
\end{Proposition}

\begin{Lemma}[Mantel \cite{Mantel1907}]\label{L2.6}
If $G$ is an $n$-vertex triangle-free graph, then $|E(G)| \le \lfloor \frac{n^{2}}{4} \rfloor$, with equality if and only if $G$ is the balanced complete bipartite graph $K_{\lfloor \frac{n}{2} \rfloor,\lceil \frac{n}{2} \rceil}$.
\end{Lemma}

As an immediate consequence of Mantel's theorem, we obtain the following bound for regular triangle-free graphs. Let $G$ be an $n$-vertex connected triangle-free $\Delta$-regular Class 1 graph. Since $G$ is $\Delta$-regular, $|E(G)| = \frac{n\Delta}{2}$. Together with Mantel's theorem and the fact that $n$ is even, we obtain $\Delta \le \frac{n}{2}$. If equality holds, then Mantel's theorem yields $G \cong K_{\frac{n}{2}, \frac{n}{2}}$, so $G$ is bipartite. Together with Lemma \ref{L2.4}, we have the following consequence.

\begin{Proposition}
Let $G$ be an $n$-vertex  triangle-free $\Delta$-regular Class 1 graph. Then $\Delta \le \frac{n}{2}$. Moreover, if $\Delta=\frac{n}{2}$, then $G \cong K_{\frac{n}{2},\frac{n}{2}}$ and every graph obtained from $G$ by a vertex-splitting is $\Delta$-critical.
\end{Proposition}

\begin{Lemma}[\cite {Häggkvist1982}] \label{L2.7}
Every $n$-vertex triangle-free graph with  $\delta>\lfloor \frac{3n}{8} \rfloor$ is homomorphic to $C_{5}$.
\end{Lemma}

\begin{Lemma}[\cite {Brandt1997}] \label{L2.8}
Let $G$ be an $n$-vertex triangle-free graph with $\delta>\frac{n}{3}$. Then $G$ is not homomorphic to $C_{5}$ if and only if $G$ contains the Wagner graph $\Gamma$ as a subgraph.
\end{Lemma}
	
To state the next lemma, we first introduce some notation and terminology concerning graph detachments, following Bahmanian and Rodger \cite {BR2012}. We emphasize that the notion of edge-coloring used in the next paragraph is different from the proper edge-coloring considered elsewhere in this paper. Here, a {\bf $k$-edge-coloring} is simply an assignment of $k$ colors to the edges, and adjacent edges are allowed to receive the same color. Let $H$ be a multigraph. An {\bf $\eta$-detachment} of $H$ is a graph $G$ obtained by replacing each vertex $w \in V(H)$ by $\eta(w)$ vertices. More precisely, there is a surjective map $\psi:V(G) \rightarrow V(H)$, called the {\bf amalgamation function}, such that $|\psi^{-1}(w)|=\eta(w)$ for every $w \in V(H)$. We write $\ell_H(w)$ for the number of loops at $w$, $m_H(w,z)$ for the number of edges joining $w$ and $z$, where $w, z \in V(H)$, $\omega(H)$ for the number of connected components of $H$, $H(i)$ for the subgraph of $H$ induced by the edges colored $i$, and $d_H(v)$ for the degree of a vertex $v$, where parallel edges contribute their multiplicity to the degree and each loop adds 2 to the degree. For real numbers $x$ and $y$, we write $x \approx y$ if $\lfloor y \rfloor \le x \le  \lceil y \rceil$. $N_{+}$ denotes the set of positive integers. We will use the following consequence of Bahmanian and Rodger \cite{BR2012} in the proof of Lemma \ref{L4.3}.

\begin{Lemma}[\cite{BR2012}]\label{L2.9}
Let $H$ be a $k$-edge-colored graph and let $\eta$ be a function from $V(H)$ into $N_{+}$ such that for each $w \in V(H)$, $\eta(w)=1$ implies $\ell_{H}(w)=0$. Then there exists a loopless $\eta$-detachment $G$ of $H$ with amalgamation function $\psi:V(G) \rightarrow V(H)$, $\eta$ being the number function associated with $\psi$, such that $G$ satisfies the following conditions:
\item [{$\mathrm{(1)}$}] $d_{G(i)}(u) \approx d_{H(i)}(w)/\eta(w)$ for each $w \in V(H)$, each $u \in \psi^{-1}(w)$, and each $i \in [1,k]$;    
\item [{$\mathrm{(2)}$}] $m_{G}(u,u^{'}) \approx \ell_{H}(w)/\binom{\eta(w)}{2}$ for each $w \in V(H)$ with $\eta(w) \ge 2$ and every pair of distinct vertices $u,u^{'} \in \psi^{-1}(w)$;
\item [{$\mathrm{(3)}$}] $m_{G}(u,v) \approx m_{H}(w,z)/(\eta(w)\eta(z))$ for every pair of distinct vertices $w,z \in V(H)$, each $u \in \psi^{-1}(w)$, and each $v \in \psi^{-1}(z)$;  
\item [{$\mathrm{(4)}$}] For each $i \in [1,k]$ such that $d_{H(i)}(w)/\eta(w)$ is an even integer for every $w \in V(H)$, the detachment satisfies $\omega(G(i))=\omega(H(i))$.     
\end{Lemma}
	
\section{Proof of Theorem 1.1}
\begin{proof}
In this proof, we denote $\Delta(G)$ by $\Delta$. It is easy to see that a regular Class 1 graph has even order, and every graph obtained from a regular graph of even order by a vertex-splitting is overfull. Thus, $G^*$ is overfull, and so $G^*$ is of Class 2. Therefore, it remains to prove that every edge of $G^\ast$ is critical. 
		
Suppose to the contrary that there exists $xy \in E(G^\ast)$ such that $xy$ is not a critical edge of $G^\ast$. Let $H=G^\ast-xy$. Then $\chi'(H)=\Delta+1$. Since $\Delta(H) \le \Delta$, Vizing's theorem implies that $\Delta(H)=\Delta$. By Lemma \ref{L2.2}, every edge incident with $a$ or $b$ is critical in $G^*$. Hence, $\{a,b\} \cap \{x,y\}=\emptyset$. Now we have $d_{H}(a)=d_{G^\ast}(a)<\Delta$, $d_{H}(b)=d_{G^\ast}(b)<\Delta$, and $d_{H}(x)=d_{H}(y)=\Delta-1$. Since $d_{G^\ast}(a)+d_{G^\ast}(b)=\Delta+2$, $(a,b)$ is a full-deficiency pair in $G^\ast$. Therefore, $(a,b)$ remains a full-deficiency pair in $H$. We know that $ab$ is a critical edge of $G^\ast$, and hence $\chi'(G^\ast-ab)=\Delta$. Since $H-ab \subseteq G^\ast-ab$, we have $\Delta \le \chi'(H-ab) \le \chi'(G^\ast-ab) = \Delta$, which implies that $\chi'(H-ab)=\Delta$ and $ab$ is also a critical edge of $H$. Thus, we have $\operatorname{dist}_{H}(x,\{a,b\}) \ge 3$ and $\operatorname{dist}_{H}(y,\{a,b\}) \ge 3$ by Lemma \ref{L2.1}. Let $T:=V(H) \setminus (N_H[a] \cup N_H[b])$. Since $G$ is triangle-free, $G^*$ is triangle-free, and consequently $H$ is triangle-free. Together with $ab \in E(H)$, this implies that $(N_H(a) \setminus \{b\}) \cap (N_H(b) \setminus \{a\})=\emptyset$. Hence $|N_H[a] \cup N_H[b]|=d_H(a)+d_H(b)=\Delta+2$. Since $|V(H)|=|V(G^*)|=n+1$, it follows that $|T|=n-1-\Delta$.

We claim that $N_{H}[x] \cup N_{H}[y] \subseteq T$; otherwise, this would contradict $\operatorname{dist}_{H}(x,\{a,b\}) \ge 3$ and $\operatorname{dist}_{H}(y,\{a,b\}) \ge 3$. Since $G^*$ is triangle-free and $xy \in E(G^*)$, this gives $N_{G^*}(x) \cap N_{G^*}(y)=\emptyset$, and hence $N_H(x) \cap N_H(y)=\emptyset$. Moreover, $xy \notin E(H)$. Hence, the sets $\{x, y\}$, $N_H(x)$, and $N_H(y)$ are pairwise disjoint. Recall that $N_{H}[x] \cup N_{H}[y] \subseteq T$, so we have $|T| \ge |N_{H}[x] \cup N_{H}[y]|=|N_{H}[x]|+|N_{H}[y]|=2+|N_{H}(x)|+|N_{H}(y)|=2+2(\Delta-1)=2\Delta$. Together with $|T|=n-1-\Delta$, we obtain $\Delta(G) \le \frac{n-1}{3}$, contradicting the condition $\Delta(G)>\frac{n-1}{3}$. This completes the proof.             
\end{proof}
	
\section{Extreme vertex-splitting}
In this section, we develop the structural and Hamilton-decomposition tools required to handle extreme vertex-splittings. We first determine where a potentially noncritical edge could lie, then show that such an edge is critical whenever the original graph has a suitable Hamilton decomposition, and finally combine structural properties of graphs homomorphic to $C_5$ with the detachment method to construct the required Hamilton decompositions. These results will be applied in the proofs of Theorems \ref{T1.2}-\ref{T1.4}.

We begin by fixing notation for an extreme vertex-splitting. Let $G$ be an $n$-vertex connected triangle-free $\Delta$-regular graph, and $G^*$ be obtained from $G$ by an extreme vertex-splitting: a vertex $b \in V(G)$ is split into adjacent vertices $a$ and $b$ such that $d_{G^*}(a)=2$ and $d_{G^*}(b)=\Delta$. Equivalently, for some $c \in N_G(b)$, $G^*$ is obtained from $G$ by replacing $bc$ with the path $(b,ba,a,ac,c)$. Thus, $d_{G^*}(v)=\Delta$ for every $v \in V(G^*) \setminus \{a\}$. Set $B:=N_{G^*}(b) \setminus \{a\}$ and $C:=N_{G^*}(c) \setminus \{a\}$. Then $V(G^*)=V(G) \cup \{a\}$ and $N_{G^*}[b] \cup N_{G^*}[c]=\{a,b,c\} \cup B \cup C$, so $R:=V(G^*) \setminus (N_{G^*}[b] \cup N_{G^*}[c])=V(G) \setminus (\{b,c\} \cup B \cup C)$. Consequently, $G^*[R] = G[R]$. Notice that $B$ and $C$ are disjoint, since $G$ is triangle-free.

\begin{Lemma}\label{L4.1}
Let $G$ be a triangle-free $\Delta$-regular Class 1 graph. Let $G^\ast$ be obtained from $G$ by an extreme vertex-splitting. If $xy$ is not a critical edge of $G^\ast$, then $xy \in E(G^\ast[R])$.
\end{Lemma}

\begin{proof}
By Lemma \ref{L2.2}, $\{x,y\} \cap \{a,b\}=\emptyset$ since $xy$ is not a critical edge. Let $H=G^*-xy$. Since $xy$ is not a critical edge and $G^\ast$ is of Class 2, $\chi'(H)=\chi'(G^*)=\Delta+1$. As $\Delta(H) \leq \Delta$, Vizing's theorem implies that $\Delta(H)=\Delta$. Moreover, $d_H(x)=d_H(y)=\Delta-1$. Since $xy$ is incident with neither $a$ nor $b$, we have $d_H(a)+d_H(b)=d_{G^*}(a)+d_{G^*}(b)=\Delta+2$. Thus, $(a,b)$ is a full-deficiency pair of $H$. We also claim that $ab$ is a critical edge of $H$. Indeed, $ab$ is a critical edge of $G^*$, and hence $G^*-ab$ is of Class 1. Since $H-ab \subseteq G^*-ab$, we have $\Delta \le \chi'(H-ab) \le \chi'(G^\ast-ab)=\Delta$, which implies that $\chi'(H-ab)=\Delta$ and $ab$ is also a critical edge of $H$.

Since the vertex-splitting is extreme, $a$ has precisely two neighbors $b$ and $c$, and hence $(N_H(a) \cup N_H(b)) \setminus \{a,b\}=\{c\} \cup B$. It follows that every vertex in $\{c\} \cup B$ has degree $\Delta$ in $H$ by Lemma \ref{L2.1}(1). Since $d_H(x)=d_H(y)=\Delta-1$, $x,y \notin \{c\} \cup B$. In particular, $d_H(c)=\Delta$. Since $d_H(a)=2$, we have $d_H(a)+d_H(c)=\Delta+2$, and hence $(a,c)$ is also a full-deficiency pair of $H$. Furthermore, $ac$ is a critical edge of $H$. Indeed, $ac$ is incident with $a$, so $ac$ is a critical edge of $G^*$ by Lemma \ref{L2.2}; consequently, $\Delta \le \chi'(H-ac) \le \chi'(G^*-ac) = \Delta<\chi'(H)$. By Lemma \ref{L2.1}(1), we obtain that every vertex in $(N_H(a) \cup N_H(c)) \setminus \{a,c\}=\{b\} \cup C$ has degree $\Delta$ in $H$. Hence, $x,y \notin C$ since $d_H(x)=d_H(y)=\Delta-1$. Note that $\{x,y\} \cap \{a,b\}=\emptyset$. Now we get $x,y \notin \{a,b,c\} \cup B \cup C$. Since $R = V(G^*) \setminus (N_{G^*}[b] \cup N_{G^*}[c])$, we have $x,y \in R$. Therefore, $xy \in E(G^*[R])$. This completes the proof. 
\end{proof}

\begin{Lemma}\label{L4.2}
Let $G$ be a $2r$-regular graph of even order $n \ge 4$ admitting a Hamilton decomposition $H_1,\dots,H_r$. Let $bc$ and $xy$ be two edges of $H_1$. Suppose that $\{b,c\} \cap \{x,y\}=\emptyset$, and let $G^*$ be obtained from $G$ by subdividing the edge $bc$ with a new vertex $a$; equivalently, replacing $bc$ by the path $(b,ba,a,ac,c)$. Then $\chi'(G^*-xy)=2r$ and $xy$ is a critical edge of $G^*$.
\end{Lemma}

\begin{proof}
For each $i \in [1,r]$, color the edges of the Hamilton cycle $H_i$ alternately with the colors $2i-1$ and $2i$, and assume that $bc$ receives color 1. Delete the edges $bc$ and $xy$. Since $\{b,c\} \cap \{x,y\} = \emptyset$ and $bc,xy \in H_1$, the graph $H_1-\{bc,xy\}$ consists of two paths, and $b$ and $c$ lie in different paths. Let $P$ be the path containing $b$. Interchange colors 1 and 2 on $P$. Color 2 is now missing at $b$, while color 1 remains missing at $c$, where a color is missing at a vertex if it is in $[1, 2r]$ but not assigned to any edge incident with that vertex. Assign colors 2 and 1 to $ab$ and $ac$, respectively. This gives a proper $2r$-edge-coloring of $G^* - xy$. Since $bc$ and $xy$ are vertex-disjoint, $d_{G^*-xy}(b)=2r$ and $\Delta(G^*-xy)=2r$, and hence $\chi'(G^*-xy)=2r$. Since $|E(G)|=rn$ where $n$ is the order, $|E(G^*)|=rn+1>rn=2r \lfloor \frac{n+1}{2} \rfloor=\Delta \lfloor \frac{|V(G^*)|}{2} \rfloor$. Hence, $G^*$ is overfull, so it is of Class 2. Consequently, $xy$ is a critical edge of $G^*$. This completes the proof. 
\end{proof}

\begin{Lemma}\label{L4.31}
Let $G$ be an $n$-vertex $\Delta$-regular non-bipartite graph, and suppose that $G$ is homomorphic to $C_5$. Let $f:V(G) \rightarrow V(C_5)$, $V_i=f^{-1}(i)$ and $n_i=|V_i|$ for $i \in [0,4]$, where all subscripts are taken modulo 5. Define $p_i:=n_{i-1}+n_{i+1}-\Delta$ and $q_i:=n_in_{i+1}-|E(G[V_i \cup V_{i+1}])|$. Then, for every $i \in [0,4]$, $n_i>0$, $p_i \ge 0$, $q_{i-1}+q_i=n_ip_i$, and if $p_i>0$, then $p_{i-1}>0$ or $p_{i+1}>0$. Moreover, $\sum_{i=0}^{4}p_i=2n-5\Delta$.
\end{Lemma}
\begin{proof}
We first show that $n_i>0$ for every $i \in [0,4]$. Suppose that $V_i=\emptyset$ for some $i$. Then the image of $f$ is contained in a proper subgraph of $C_5$. Since every proper subgraph of $C_5$ is bipartite, $G$ is bipartite, contradicting the fact that $G$ is non-bipartite. Since $G$ is homomorphic to $C_5$, $V_i$ is an independent set for every $i \in [0,4]$. Moreover, every edge of $G$ has its endvertices in two consecutive parts. Thus, every vertex $v \in V_i$ has all its neighbors in $V_{i-1} \cup V_{i+1}$. Since $G$ is $\Delta$-regular, $\Delta=d_G(v) \le n_{i-1}+n_{i+1}$. Hence, $p_i=n_{i-1}+n_{i+1}-\Delta \ge 0$. For each fixed vertex $v \in V_i$, the number of vertices in $V_{i-1} \cup V_{i+1}$ that are not adjacent to $v$ is $p_i=n_{i-1}+n_{i+1}-\Delta$. Hence, the total number of missing adjacencies between $V_i$ and $V_{i-1} \cup V_{i+1}$ is $n_ip_i$. On the other hand, the number of missing edges between $V_{i-1}$ and $V_i$ is $q_{i-1}$, while the number of missing edges between $V_i$ and $V_{i+1}$ is $q_i$. Therefore, $q_{i} \ge 0$ and $q_{i-1}+q_i=n_ip_i$. We next show that, if $p_i>0$, then $p_{i-1}>0$ or $p_{i+1}>0$. Otherwise, from $q_{i-1}+q_i=n_ip_i$ we get $q_{i-2}+q_{i-1}=0$ and $q_i+q_{i+1}=0$. Thus, $q_{i-1}=q_i=0$ and $n_ip_i=q_{i-1}+q_i=0$, contradicting $n_i>0$ and $p_i>0$. Finally, $\sum_{i=0}^{4}p_i=\sum_{i=0}^{4}(n_{i-1}+n_{i+1}-\Delta)=2\sum_{i=0}^{4}n_i-5\Delta=2n-5\Delta$. This completes the proof. 
\end{proof}

\begin{Lemma}\label{L4.3}
Let $r\ge2$ and $G$ be a $2r$-regular graph with $V(G)=V_0\dot{\cup}\cdots\dot{\cup}V_4$, where $V_i$ for $i \in [0, 4]$ is an independent set, $|V_0|=|V_1|=r+1$, $|V_2|=|V_4|=r$ and $|V_3|=r-1$. Suppose that $G[V_0 \cup V_1] \cong K_{r+1,r+1} - M$ for a perfect matching $M$, every other consecutive pair $V_i$ and $V_{i+1}$ induces a complete bipartite graph, and there are no edges between nonconsecutive parts.
Let $F_1,\dots,F_r$ be a 1-factorization of $G[V_0 \cup V_1]$. Then $G$ admits a Hamilton decomposition $C_1,\dots,C_r$ such that $E(C_j) \cap E(G[V_0 \cup V_1])=F_j$ for $j \in [1,r]$.
\end{Lemma}

\begin{proof}
We construct an $r$-edge-colored multigraph $J$ by amalgamating the vertices in $V_2$, $V_3$, $V_4$, while leaving all vertices in $V_0 \cup V_1$ unchanged. Let $V(J)=V_0 \cup V_1 \cup \{w_2,w_3,w_4\}$, where $w_i$ corresponds to the amalgamation of all vertices in $V_i$, for $i \in [2,4]$. We use the color set $[1,r]$. For each $j \in [1,r]$, define the edges of color $j$ in $J$ as follows:
\begin{itemize}
\item assign color $j$ to every edge of the 1-factor $F_j$;
\item add one edge $vw_2$ of color $j$ for every $v \in V_1$;
\item add $r-1$ parallel edges between $w_2$ and $w_3$ and color them all with color $j$;
\item add $r-1$ parallel edges between $w_3$ and $w_4$ and color them all with color $j$;
\item add one edge $w_4u$ of color $j$ for every $u \in V_0$.
\end{itemize}

Notice that the edge-coloring used here is the coloring notion in Lemma \ref{L2.9}: it is merely an assignment of colors to edges and is not required to be proper. Define $\eta:V(J) \rightarrow N_+$ by $\eta(u)=1$ for $u \in V_0 \cup V_1$, $\eta(w_2)=r$, $\eta(w_3)=r-1$, and $\eta(w_4)=r$. Since $r \ge 2$, all values of $\eta$ are positive integers. Moreover, since $J$ is loopless, the conditions of Lemma \ref{L2.9} are satisfied. Let $\widetilde{G}$ be the resulting loopless $\eta$-detachment, with amalgamation function $\psi:V(\widetilde{G}) \rightarrow V(J)$. We identify the fibers $\psi^{-1}(w_2)$, $\psi^{-1}(w_3)$, $\psi^{-1}(w_4)$ with $V_2,V_3,V_4$, respectively. The edges of $\widetilde{G}$ and $J$ correspond bijectively, with corresponding endvertices mapped by $\psi$. When the edges are colored, corresponding edges retain their colors.
 
Claim 1: $\widetilde{G} \cong G$.
\begin{proof}
For $i=2,4$, we have $\eta(w_i)=r \ge 2$, and hence Lemma \ref{L2.9}(2) gives $E(\widetilde{G}[V_i])=\emptyset$. If $r \ge 3$, the same argument applies to $V_3$. If $r=2$, then $|V_3|=1$, and hence we still have $E(\widetilde{G}[V_3])=\emptyset$. Together with the fact that the sets $V_0$ and $V_1$ are not detached, this implies that all $V_i$ with $i \in [0,4]$ remain independent. Now, we consider the edges between consecutive parts.

Fix $v \in V_1$. For each of the $r$ colors, there is exactly one edge joining $v$ to $w_2$. Hence, $m_{J}(v,w_2)=r$. Since $\eta(v)=1$ and $\eta(w_2)=r$, we have that for every $z \in V_2=\psi^{-1}(w_2)$, $m_{\widetilde{G}}(v,z) \approx \frac{m_{J}(v,w_2)}{\eta(v)\eta(w_2)}=\frac{r}{1 \cdot r}=1$ by Lemma \ref{L2.9}(3). Therefore, $\widetilde{G}[V_1 \cup V_2] \cong K_{r+1,r}$. The same argument yields $\widetilde{G}[V_4 \cup V_0] \cong K_{r,r+1}$. There are $r-1$ edges of each of the $r$ colors between $w_2$ and $w_3$. Thus, $m_{J}(w_2,w_3)=r(r-1)$. For every $z_2 \in V_2 = \psi^{-1}(w_2)$ and $z_3 \in V_3=\psi^{-1}(w_3)$, we get $m_{\widetilde{G}}(z_2,z_3) \approx \frac{r(r-1)}{r(r-1)}=1$ by Lemma \ref{L2.9}(3). Hence, $\widetilde{G}[V_2 \cup V_3] \cong K_{r,r-1}$. The same argument yields $\widetilde{G}[V_3 \cup V_4] \cong K_{r-1,r}$. The vertices of $V_0 \cup V_1$ are not detached. Since $F_1,\dots,F_r$ form a 1-factorization of $G[V_0 \cup V_1]$, the edges between $V_0$ and $V_1$ in $J$, and hence in $\widetilde{G}$, are exactly the original edges of $G[V_0 \cup V_1]$. Therefore, $\widetilde{G}[V_0 \cup V_1] \cong K_{r+1,r+1}-M$. Finally, if two vertices of $J$ corresponding to nonconsecutive parts are not joined by an edge, then no additional edges are created by Lemma \ref{L2.9}(3).

Since each $V_i$ with $i \in [0,4]$ remains independent and the above isomorphisms hold, we have $\widetilde{G} \cong G$. In particular, $\widetilde{G}$ is simple. From now on, we identify $\widetilde{G}$ with $G$.
\end{proof}

Claim 2: For every $j \in [1,r]$, the color-$j$ subgraph $G(j)$ is a spanning 2-regular subgraph of $G$.
\begin{proof}
Fix $j \in [1,r]$. We compute the degree of each vertex in $J(j)$. For every $u \in V_0$, exactly one edge of $F_j$ is incident with $u$, and there is exactly one edge assigned color $j$ joining $u$ to $w_4$. Hence, $d_{J(j)}(u)=2$. Similarly, for every $v \in V_1$, $d_{J(j)}(v)=2$. At the amalgamated vertices, $d_{J(j)}(w_2)=|V_1|+(r-1)=(r+1)+(r-1)=2r$, $d_{J(j)}(w_3)=(r-1)+(r-1)=2(r-1)$, and $d_{J(j)}(w_4)=(r-1)+|V_0|=(r-1)+(r+1)=2r$. Consequently, $\frac{d_{J(j)}(u)}{\eta(u)}=2$ for $u \in V_0 \cup V_1$, $\frac{d_{J(j)}(w_2)}{\eta(w_2)}=\frac{2r}{r}=2$, $\frac{d_{J(j)}(w_3)}{\eta(w_3)}=\frac{2(r-1)}{r-1}=2$ and $\frac{d_{J(j)}(w_4)}{\eta(w_4)}=\frac{2r}{r}=2$. Thus, for every $w \in V(J)$, $\frac{d_{J(j)}(w)}{\eta(w)}=2$. By Lemma \ref{L2.9}(1), every vertex $u \in V(G)$ satisfies $d_{G(j)}(u)=2$. In particular, every vertex of $G$ is incident with an edge assigned color $j$, and hence $V(G(j))=V(G)$. Hence, $G(j)$ is a spanning 2-regular subgraph of $G$. 
\end{proof}
At this stage, $G(j)$ may be a disjoint union of cycles. It remains to prove that it is connected. For every $j \in [1,r]$, $J(j)$ is connected. Indeed, every vertex of $V_1$ is adjacent to $w_2$, which is in turn adjacent to $w_3$; $w_3$ is adjacent to $w_4$, and every vertex of $V_0$ is adjacent to $w_4$. Furthermore, $F_j$ is a perfect matching between $V_0$ and $V_1$. Since $r \ge 2$, there is at least one edge of color $j$ between $w_2$ and $w_3$, as well as between $w_3$ and $w_4$. Therefore, $\omega(J(j))=1$. Moreover, $\frac{d_{J(j)}(w)}{\eta(w)}=2$ is an even integer for every $w \in V(J)$. By Lemma \ref{L2.9}(4), we get $\omega(G(j))=\omega(J(j))=1$. Combining Claim 2 with the above equality, $G(j)$ is a connected spanning 2-regular graph. Hence, $G(j)$ is a Hamilton cycle. Define $C_j:=G(j)$ for $j \in [1,r]$. Since the color classes partition $E(G)$, we have $E(G)=E(C_1)\dot\cup\cdots\dot\cup E(C_r)$, and therefore $C_1,\dots,C_r$ form a Hamilton decomposition of $G$.

Finally, the vertices of $V_0 \cup V_1$ are not detached. By construction, the edges of color $j$ between $V_0$ and $V_1$ are precisely the edges of $F_j$. Hence, for every $j \in [1,r]$, $E(C_j) \cap E(G[V_0 \cup V_1])=F_j$. This completes the proof. 
\end{proof}

\begin{Lemma}\label{L4.4}
Let $G$ be an $n$-vertex triangle-free $\Delta$-regular non-bipartite Class 1 graph. Suppose that $G$ is homomorphic to $C_5$ and $2n-5\Delta=2$. Then $\Delta$ is even. Let $r:=\frac{\Delta}{2}$. For an edge $bc \in E(G)$, let $B=N_G(b) \setminus \{c\}$, $C=N_G(c) \setminus \{b\}$, and $R=V(G) \setminus (N_G[b] \cup N_G[c])$. If $xy \in E(G[R])$, then $G$ admits a Hamilton decomposition $C_1,\dots,C_r$ such that $bc,xy \in E(C_j)$ for some $j \in [1,r]$.
\end{Lemma}

\begin{proof}
Let $V_i,n_i,p_i$ and $q_i$ be as in Lemma \ref{L4.31} for $i \in [0,4]$. By Lemma \ref{L4.31}, we get $\sum_{i=0}^{4}p_i=2$, $p_i \ge 0$, and $n_i>0$ for $i \in [0,4]$. Now, we determine the structure of $G$.

Without loss of generality, $p_0=p_1=1$ and $p_2=p_3=p_4=0$. Indeed, Lemma \ref{L4.31} shows that if $p_i>0$, then $p_{i-1}>0$ or $p_{i+1}>0$. Since $\sum_{i=0}^{4}p_i=2$, either one of the numbers $p_i$ with $i \in [0,4]$ equals 2, or exactly two of them equal 1. The first possibility contradicts Lemma \ref{L4.31}. In the second possibility, the two positive terms must be consecutive. Recall $p_i=n_{i-1}+n_{i+1}-\Delta$. From $p_2=p_3=p_4=0$, we have $n_1+n_3=\Delta$, $n_2+n_4=\Delta$, and $n_3+n_0=\Delta$. Also, $p_0=p_1=1$ gives $n_4+n_1=\Delta+1$ and $n_0+n_2=\Delta+1$. Comparing $n_1+n_3=\Delta$ with $n_3+n_0=\Delta$ yields $n_0=n_1$. Similarly, from $n_4+n_1=\Delta+1$ and $n_0+n_2=\Delta+1$, we get $n_2=n_4$. Hence, $2n_2=n_2+n_4=\Delta$. It follows that $\Delta$ is even. Let $r=\Delta/2$. Therefore, we have $n_0=n_1=r+1$, $n_2=n_4=r$, and $n_3=r-1$. In particular, $n=5r+1$. Since $G$ is a regular Class 1 graph, $n$ is even. Then $n=5r+1$ implies that $r$ is odd. Moreover, $r \ge 3$ since $n_3=r-1 \ge 1$.

We now determine the edges between consecutive parts. From $q_{i-1}+q_i=n_ip_i$ and $p_2=p_3=p_4=0$, we get $q_1=q_2=q_3=q_4=0$. Thus, $G[V_1 \cup V_2]$, $G[V_2 \cup V_3]$, $G[V_3 \cup V_4]$, and $G[V_4 \cup V_0]$ are complete bipartite graphs. Furthermore, from $q_{i-1}+q_i=n_ip_i$ with $i=0$, we have $q_0=n_0p_0=r+1$. For every $u \in V_0$, $|V_4|+|V_1|=r+(r+1)=2r+1=\Delta+1$. Since $d_{G}(u)=\Delta$ and $u$ is adjacent to every vertex of $V_4$, exactly one vertex in $V_1$ is not adjacent to $u$. Similarly, for every vertex of $V_1$, exactly one vertex in $V_0$ is not adjacent to it. Hence, the missing edges between $V_0$ and $V_1$ form a perfect matching $M$. Therefore, $G[V_0 \cup V_1] \cong K_{r+1,r+1}-M$.

We now locate the edge $bc$. By the symmetry of the five-part structure, it suffices to consider the following three cases. First, suppose that $bc \in E(G[V_2 \cup V_3])$. Since $G[V_1 \cup V_2]$, $G[V_2 \cup V_3]$ and $G[V_3 \cup V_4]$ are complete bipartite graphs, we have $R=V_0$. But $V_0$ is independent, so $E(G[R])=\emptyset$, contradicting the fact $xy \in E(G[R])$. Thus, this case cannot occur, and it remains to consider the following two cases.

Case 1. $bc \in E(G[V_0 \cup V_1])$.

Let $V_0=\{u_0,\dots,u_r\}$, $V_1=\{v_0,\dots,v_r\}$, and $M=\{u_iv_i:i \in [0,r]\}$ after relabeling. Since $bc \in E(G[V_0 \cup V_1])$, we may write $bc=u_iv_j$ with $i,j \in [0,r]$ and $i \neq j$. Now we get $R=V_3 \cup \{u_j,v_i\}$. Note that the set $V_3$ is independent and has no neighbors in $\{u_j,v_i\}$, while $u_jv_i \in E(G)$ because $i \neq j$. Hence, $E(G[R])=\{u_jv_i\}$ and $xy=u_jv_i$. The two edges $u_iv_j,u_jv_i$ are disjoint. Excluding their endvertices in $V_0$ and $V_1$, each of the two remaining parts contains $r-1 \ge 2$ vertices. Let $V_{0}'=\{u_{s_1},u_{s_2},\dots,u_{s_{r-1}}\}$ and $V_{1}'=\{v_{s_1},v_{s_2},\dots,v_{s_{r-1}}\}$ be the sets of the remaining vertices. Thus, after relabeling, there exists a perfect matching $F_1=\{u_iv_j,u_jv_i,u_{s_1}v_{s_2},u_{s_2}v_{s_3},\dots,u_{s_{r-1}}v_{s_1}\}$ of $G[V_0 \cup V_1]$ containing both $bc=u_iv_j$ and $xy=u_jv_i$. Since $G[V_0 \cup V_1]-F_1$ is an $(r-1)$-regular bipartite graph, it has a 1-factorization $F_2,\dots,F_r$. Together with $F_1$, this gives a 1-factorization $F_1,\dots,F_r$ of $G[V_0 \cup V_1]$. By Lemma \ref{L4.3}, $G$ admits a Hamilton decomposition $C_1,\dots,C_r$ such that $E(C_1) \cap E(G[V_0 \cup V_1])=F_1$. Then we get $bc,xy \in E(C_1)$.

Case 2. $bc \in E(G[V_1 \cup V_2])$.

We may assume without loss of generality that $b=v_i \in V_1$ and $c \in V_2$ for some $i \in [0,r]$, where $u_i \in V_0$ is the unique vertex not adjacent to $v_i$. Now we get $R=V_4 \cup \{u_i\}$. Thus, $xy = u_iz$ for some $z \in V_4$ since $V_4$ is independent and $u_i$ is adjacent to every vertex of $V_4$. Since $G[V_0 \cup V_1]$ is an $r$-regular bipartite graph, let $F_1,\dots,F_r$ be a 1-factorization of $G[V_0 \cup V_1]$. By Lemma 4.4, $G$ admits a Hamilton decomposition $C_1,\dots,C_r$ such that $E(C_j) \cap E(G[V_0 \cup V_1])=F_j$ for every $j \in [1,r]$. Consequently, each vertex of $V_0 \cup V_1$ is incident with exactly one edge of $F_j$ in $C_j$. We have that the other edge of $C_j$ at $v_i$ joins it to some $c' \in V_2$, and that the other edge of $C_j$ at $u_i$ joins it to some $z' \in V_4$. All vertices of $V_2$ have neighborhood $V_1 \cup V_3$, and all vertices of $V_4$ have neighborhood $V_3 \cup V_0$. Hence, we may assume without loss of generality that $c'=c$ and $z'=z$. Then $bc,xy \in E(C_j)$ for some $j \in [1,r]$, where $bc=v_ic$ and $xy=u_iz$. This completes the proof. 
\end{proof}

\section{Proof of Theorem 1.2}
\begin{proof}
Since $G$ is a regular Class 1 graph, $n$ is even. Moreover, $G^*$ is overfull, and so $G^*$ is of Class 2. Moreover, $|E(G)|=3n$. Since $G$ is triangle-free, $3n \le \frac{n^2}{4}$ by Lemma \ref{L2.6}. Hence, $n \ge 12$. On the other hand, $6>\frac{n}{3}$ implies $n<18$. Therefore, $n \in \{12,14,16\}$. If $n \in \{12,14\}$, then $6>\frac{2n}{5}$, and hence $G^\ast$ is $6$-critical by Proposition \ref{C2.1}. Thus, it remains to consider the case $n=16$. If the vertex-splitting is non-extreme, then $G^\ast$ is $6$-critical since $6=\Delta(G)>\frac{n-1}{3}=\frac{16-1}{3}$ by Theorem \ref{T1.1}.

We may therefore assume that the vertex-splitting is extreme. Equivalently, $G^*$ is obtained from $G$ by subdividing some edge $bc \in E(G)$ with a new vertex $a$, replacing $bc$ by the path $(b,ba,a,ac,c)$. If $G$ is bipartite, then $G^*$ is 6-critical by Lemma \ref{L2.4}. Thus, we may further assume that $G$ is non-bipartite. In order to prove that $G^*$ is 6-critical, we only need to show that every edge of $G^*$ is critical. Suppose to the contrary that there exists an edge $xy \in E(G^*)$ that is not a critical edge of $G^*$. 

Let $B=N_{G^*}(b) \setminus \{a\}, C=N_{G^*}(c) \setminus \{a\}$, and $R=V(G^*) \setminus (N_{G^*}[b] \cup N_{G^*}[c])$. By Lemma \ref{L4.1}, $xy \in E(G^*[R])$. Since $a,b,c \notin R$ and $G^*$ differs from $G$ only by replacing the edge $bc$ with the path $(b,ba,a,ac,c)$, we have $G^*[R]=G[R]$, and hence $xy \in E(G[R])$. Since $\delta(G)=6>\frac{16}{3}=\frac{n}{3}$, Lemma \ref{L2.8} implies that either $G$ is homomorphic to $C_5$, or $G$ contains the Möbius ladder $\Gamma$ as a subgraph. We consider these two cases separately.

Case 1. $G$ is homomorphic to $C_5$.

In this case, $2n-5\Delta=2 \times 16-5 \times 6=2$. Since $xy \in E(G[R])$, $G$ admits a Hamilton decomposition $C_1,C_2,C_3$ such that $bc,xy \in E(C_1)$ by Lemma \ref{L4.4}, after relabeling the Hamilton cycles if necessary. By Lemma \ref{L4.2}, $xy$ is a critical edge of $G^*$, contradicting the hypothesis that $xy$ is not a critical edge.

Case 2. $G$ contains $\Gamma$ as a subgraph.

Let $X=\{x_0,x_1,\dots,x_7\}$ be the vertex set of $\Gamma$ contained in $G$, where all subscripts in this case are taken modulo 8, and $N_{\Gamma}(x_i)=\{x_{i-1},x_{i+1},x_{i+4}\}$ for $i \in [0,7]$. Let $Y=V(G) \setminus X$. Obviously, $|Y|=8$. For a vertex $z$, write $N_X(z)=N_G(z) \cap X$, $N_Y(z)=N_G(z) \cap Y$, $d_X(z)=|N_X(z)|$, and $d_Y(z)=|N_Y(z)|$. Note that any two vertices of $\Gamma$ that are not adjacent have a common neighbor in $\Gamma$. Therefore, an additional edge in $G[X]$ would create a triangle. Since $G$ is triangle-free, $G[X] \cong \Gamma$. We next determine the structure of $G[Y]$.

Claim: The vertices of $Y$ can be labeled as $Y=\{y_0,\dots,y_7\}$ so that, for every $i \in [0,7]$, $N_X(y_i)=\{x_i,x_{i+2},x_{i+5}\}$ and $N_Y(y_i)=\{y_{i-1},y_{i+1},y_{i+4}\}$. In particular, $G[Y] \cong \Gamma$.
\begin{proof}
For $i \in [0,7]$, we have $d_Y(x_i)=3$ since $G$ is 6-regular and every $x_i$ has degree 3 in $G[X]$. Consequently, $|E_G(X,Y)|=8 \times 3=24$, where $E_G(X,Y)=\{xy \in E(G):x \in X,\ y \in Y\}$. We next determine the neighborhoods of the vertices in $Y$. Since $G$ is triangle-free, we get that $N_X(y)$ is an independent set of $G[X] \cong \Gamma$ for every $y \in Y$. Thus, for every $y \in Y$, $d_X(y) \le 3$ since the independence number $\alpha(\Gamma) = 3$. On the other hand, since $|E_G(X,Y)|=24$, we have $\sum_{y \in Y}d_X(y)=24$. Together with $|Y|=8$ and $d_X(y) \le 3$ for every $y \in Y$, this implies $d_X(y)=3$, and hence $d_Y(y)=3$ since $G$ is $6$-regular. It is easy to see that the eight independent sets of size $3$ in $\Gamma$ are $S_i:= \{x_i,x_{i+2},x_{i+5}\}$, where $i \in [0,7]$. Therefore, for each $y \in Y$, the fact that $d_X(y)=3$ forces $N_X(y)=S_i$ for some $i$. Let $m_i:=\left|\{y \in Y:N_X(y)=S_i\}\right|$ for $i \in [0,7]$. For each $j \in [0,7]$, the vertex $x_j$ belongs to precisely $S_j$, $S_{j-2}$, and $S_{j-5}$. Since $d_Y(x_j)=3$, we have $m_j+m_{j-2}+m_{j-5}=3$, and replacing $j$ by $j+5$ gives $m_{j+5}+m_{j+3}+m_j=3$. Together with $j-5 \equiv j+3 \pmod 8$, it follows that $m_{j-2}=m_{j+5}=m_{j-3}$. As $j$ is arbitrary, we conclude that $m_0=m_1=\cdots=m_7$. Moreover, $\sum_{i=0}^{7}m_i=|Y|=8$, which implies $m_i=1$. Hence, for each $i \in [0,7]$, let $y_i$ denote the unique vertex of $Y$ satisfying $N_X(y_i)=S_i=\{x_i,x_{i+2},x_{i+5}\}$. Suppose that $y_iy_j \in E(G)$. Since $G$ is triangle-free,  we have $N_X(y_i) \cap N_X(y_j)=S_i \cap S_j=\emptyset$. For $i \in [0,7]$, since $S_i=\{x_i,x_{i+2},x_{i+5}\}$, it follows that $S_i \cap S_j=\emptyset$ if and only if $j-i \equiv \pm 1,4 \pmod 8$. Hence, $N_Y(y_i) \subseteq \{y_{i-1},y_{i+1},y_{i+4}\}$. As $d_Y(y_i)=3$, this forces $N_Y(y_i)=\{y_{i-1},y_{i+1},y_{i+4}\}$. This proves that $G[Y] \cong \Gamma$.
\end{proof}
We now identify the complete structure of $G$. We have $N_G(x_i)=N_G(y_{i-1})$ for $i \in [0,7]$ since $N_X(x_i)=N_X(y_{i-1})=\{x_{i-1},x_{i+1},x_{i+4}\}$ and $N_Y(x_i)=N_Y(y_{i-1})=\{y_{i-2},y_i,y_{i+3}\}$ by the Claim above. Define $P_i:= \{x_i,y_{i-1}\}$ for $i \in [0,7]$. Since $x_i \notin N_X(y_{i-1})$, each $P_i$ is an independent set of size 2. Let $i \neq j$. If $x_ix_j \in E(G[X])$, then $N_G(x_i)=N_G(y_{i-1})$ implies $y_{i-1}x_j \in E(G)$, while $N_G(x_j)=N_G(y_{j-1})$ implies $x_iy_{j-1} \in E(G)$. By applying the equality $N_G(x_i)=N_G(y_{i-1})$ again, we obtain $y_{i-1}y_{j-1} \in E(G)$. Therefore, $G[P_i \cup P_j] \cong K_{2,2}$. Conversely, if $x_ix_j \notin E(G[X])$, then $E(G[P_i \cup P_j])=\emptyset$. Note that $G[X] \cong \Gamma$. Thus, $G$ (see Figure~\ref{fig:figure_G}) is obtained from $\Gamma$ by replacing each vertex with an independent set of size $2$ and each edge with $K_{2,2}$.
\begin{figure}[htbp]
	\begin{center}
		\centering
		\scalebox{0.4}{\includegraphics{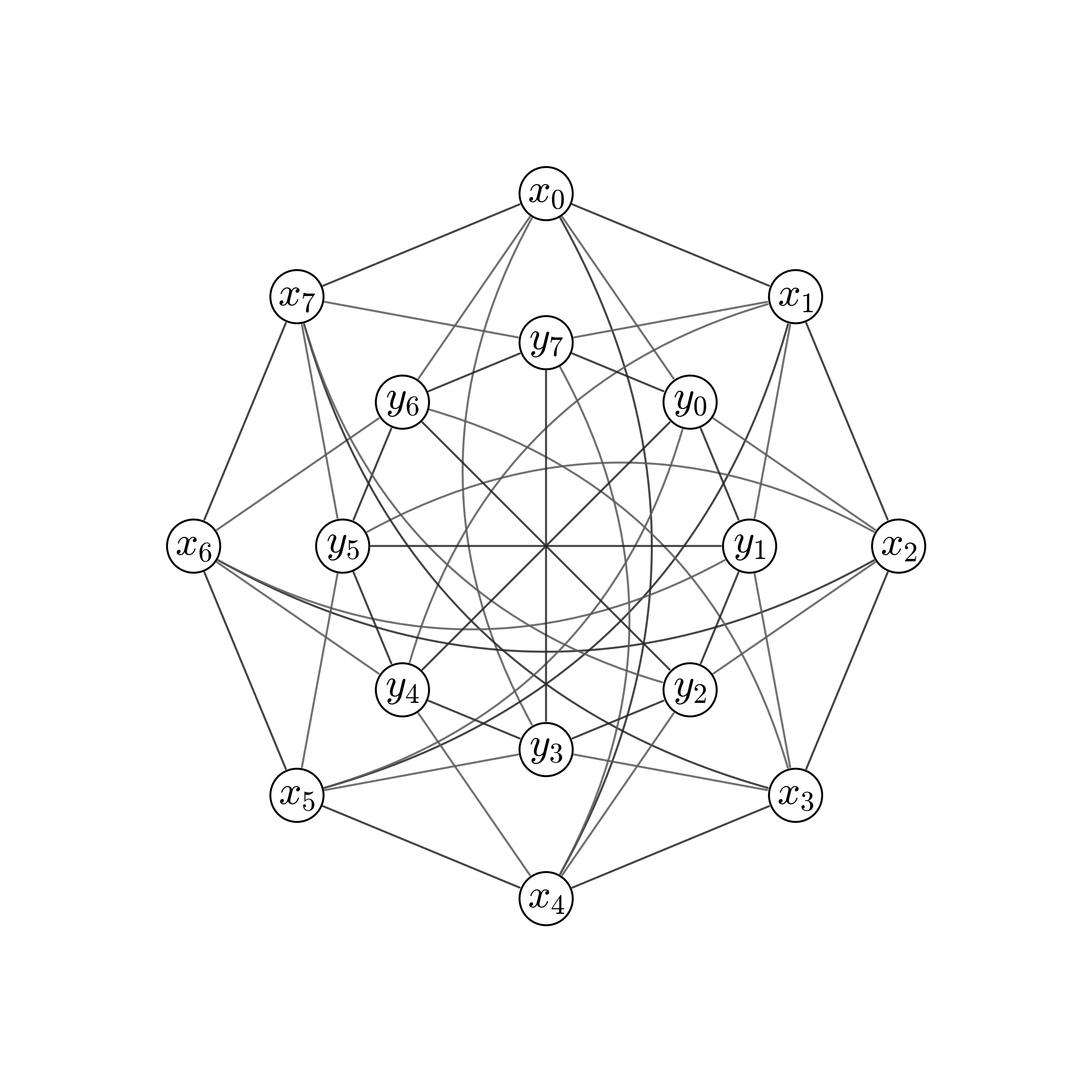}}
		\caption{The graph $G$ in  Case~2 of the proof
of Theorem \ref{T1.2}.}
        \label{fig:figure_G}
	\end{center}
\end{figure}
There are two types of edges in $\Gamma$: the edges of the outer $8$-cycle and the four opposite chords. By symmetry, it suffices to consider one representative of each type.

Subcase 2.1. $bc$ corresponds to an edge of the 8-cycle.

Suppose that $b \in P_0$ and $c \in P_1$. Write $P_0=\{b,b'\}$ and $P_1=\{c,c'\}$. Since $N_{G[X]}(x_0)=\{x_1,x_4,x_7\}$, we have $N_G(b)=P_1 \cup P_4 \cup P_7$ and hence $B=P_4 \cup P_7 \cup \{c'\}$. Similarly, $N_{G[X]}(x_1)=\{x_0,x_2,x_5\}$, so $C=P_2 \cup P_5 \cup \{b'\}$. It follows that $R=P_3 \cup P_6$. Since $x_3x_6 \notin E(G[X])$, we have $E(G[R])=E(G[P_3 \cup P_6])=\emptyset$, contradicting the fact $xy \in E(G[R])$. 

Subcase 2.2. $bc$ corresponds to an opposite edge of $\Gamma$.

Suppose that $b \in P_0$ and $c \in P_4$. Write $P_0=\{b,b'\}$ and $P_4=\{c,c'\}$. Since $N_{G[X]}(x_0) = \{x_1,x_4,x_7\}$, we have $N_G(b)=P_1 \cup P_4 \cup P_7$ and hence $B=P_1 \cup P_7 \cup \{c'\}$. Similarly, $N_{G[X]}(x_4)=\{x_0,x_3,x_5\}$, we have $N_G(c)=P_0 \cup P_3 \cup P_5$ and hence $C=P_3 \cup P_5 \cup \{b'\}$. Therefore, $R=P_2 \cup P_6$. Since $x_2x_6 \in E(G[X])$, we have $G[P_2 \cup P_6] \cong K_{2,2}$. Consequently, $G[R] \cong C_4$. Since $xy \in E(G[R])$, $xy$ is one of the four edges between $P_2$ and $P_6$. Together with the fact that the two vertices in each $P_i$ have identical neighborhoods for $i \in [0,7]$, we may assume $b=x_0$, $c=x_4$, $xy=x_2x_6$. To give a Hamilton decomposition of $G$, we now define three cycles $C_1$, $C_2$, and $C_3$ by their vertex sequences, where edges are implicitly represented by consecutive vertices. Define
\begin{equation}
	\begin{aligned}
		C_1=&(x_0,x_4,y_7,x_1,y_1,y_0,x_2,x_6,x_5,y_5,y_4,y_3,x_3,x_7,y_2,y_6,x_0),
	\end{aligned} \notag
\end{equation}
\begin{equation}
	\begin{aligned}
		C_2=&(x_0,x_1,x_2,x_3,y_1,y_2,x_4,x_5,y_0,y_4,x_6,y_6,y_5,x_7,y_7,y_3,x_0),
	\end{aligned} \notag
\end{equation}
and
\begin{equation}
	\begin{aligned}
		C_3=&(x_0,x_7,x_6,y_1,y_5,x_2,y_2,y_3,x_5,x_1,y_4,x_4,x_3,y_6,y_7,y_0,x_0).
	\end{aligned} \notag
\end{equation}

Each of $C_1$, $C_2$, and $C_3$ contains all 16 vertices exactly once and every consecutive pair is adjacent. Therefore, each of them forms a Hamilton cycle. Clearly, $E(C_1)$, $E(C_2)$, and $E(C_3)$ are pairwise disjoint. Since each Hamilton cycle contains 16 edges and $|E(G)|=\frac{16 \times 6}{2}=48$, we have $E(G)=E(C_1)\dot{\cup}E(C_2)\dot{\cup}E(C_3)$. Hence, $C_1$, $C_2$, and $C_3$ form a Hamilton decomposition of $G$. Furthermore, since $bc=x_0x_4 \in E(C_1)$ and $xy=x_2x_6 \in E(C_1)$, we have $bc,xy \in E(C_1)$. By Lemma \ref{L4.2}, $xy$ is a critical edge of $G^*$, contradicting the hypothesis that $xy$ is not a critical edge. This completes the proof.
\end{proof}

\section{Proofs of Theorems 1.3 and 1.4}

We begin with the proof of Theorem \ref{T1.3}. 
\begin{proof}
Since $G$ is a regular Class 1 graph, $n$ is even. Moreover, $G^*$ is overfull, and so $G^*$ is of Class 2. For $n>24$, we have $\Delta \ge \frac{2n-3}{5}>\frac{n-1}{3}$. Hence, if the vertex-splitting is non-extreme, then $G^*$ is $\Delta$-critical by Theorem \ref{T1.1}. We may therefore assume that the vertex-splitting is extreme. Equivalently, $G^*$ is obtained from $G$ by subdividing an edge $bc \in E(G)$ with a new vertex $a$, replacing $bc$ by the path $(b,ba,a,ac,c)$. If $G$ is bipartite, then the result follows from Lemma \ref{L2.4}. Thus, we may assume that $G$ is non-bipartite. In order to prove that $G^*$ is $\Delta$-critical, we only need to show that every edge of $G^*$ is critical. Suppose to the contrary that there exists an edge $xy \in E(G^*)$ that is not a critical edge of $G^*$.

Let $B=N_{G^*}(b) \setminus \{a\}$, $C=N_{G^*}(c) \setminus \{a\}$, and $R=V(G^*) \setminus (N_{G^*}[b] \cup N_{G^*}[c])$. By Lemma \ref{L4.1}, we have $xy \in E(G^*[R])$. Since $a,b,c \notin R$ and $G^*$ differs from $G$ only by replacing the edge $bc$ with the path $(b,ba,a,ac,c)$, we have $G^*[R]=G[R]$, and hence $xy \in E(G[R])$. Since $G$ is triangle-free and non-bipartite, Lemma \ref{L2.3} gives $\Delta \le \frac{2n}{5}$. Define $t:=2n-5\Delta$. From $\Delta \le \frac{2n}{5}$, it follows that $t \ge 0$. On the other hand, $\Delta \ge \frac{2n-3}{5}$ yields $t=2n-5\Delta \le 3$. Hence, $t \in \{0,1,2,3\}$. Since $n>24$, $\delta=\Delta \ge \frac{2n-3}{5}>\frac{3n}{8}$. Hence, Lemma \ref{L2.7} implies that $G$ is homomorphic to $C_5$. Thus, the conditions of Lemma \ref{L4.31} are satisfied, and we follow its notation. Now we have $\sum_{i=0}^{4}p_i=2n-5\Delta=t$.

By Lemma \ref{L4.31}, if $p_i>0$, then $p_{i-1}>0$ or $p_{i+1}>0$. We now consider the four possible values of $t$ to get contradictions.

Case 1. $t=0$.

Since $\sum_{i=0}^{4}p_i=t=0$ and each $p_i \ge 0$, it follows that $p_i=0$ for every $i \in [0,4]$. Then Lemma \ref{L4.31} implies that $q_{i-1}+q_i=n_ip_i=0$ for every $i$, and since $q_i \ge 0$, it follows that $q_i=0$ for every $i \in [0,4]$. Therefore, every two consecutive parts induce a complete bipartite graph. Without loss of generality, assume $b \in V_0$ and $c \in V_1$. Since $N_G(b)=V_4 \cup V_1$ and $N_G(c)=V_0 \cup V_2$, it follows that $B=(V_4 \cup V_1) \setminus \{c\}$, $C=(V_0 \cup V_2) \setminus \{b\}$, and $R=V_3$. But $V_3$ is independent, and hence $E(G[R])=\emptyset$, contradicting the fact $xy \in E(G[R])$.

Case 2. $t=1$.

Since $\sum_{i=0}^{4}p_i=t=1$, there is a unique $j \in \{0,1,2,3,4\}$ such that $p_j=1$, while $p_i=0$ for $i \neq j$. In particular, $p_{j-1}=p_{j+1}=0$, contradicting Lemma \ref{L4.31}.

Case 3. $t=2$.

In this case $2n-5\Delta=2$. Since $xy \in E(G[R])$, all the assumptions of Lemma \ref{L4.4} are satisfied. Therefore, $G$ admits a Hamilton decomposition $C_1,\dots,C_r$ such that $bc,xy \in E(C_j)$ for some $j \in [1,r]$. Without loss of generality, we may assume that $bc,xy \in E(C_1)$. By Lemma \ref{L4.2}, $xy$ is a critical edge of $G^*$, contradicting our choice of $xy$.

Case 4. $t=3$.

It remains to consider $\sum_{i=0}^{4}p_i=3$. Recall that for every $i \in [0,4]$, if $p_i>0$, then $p_{i-1}>0$ or $p_{i+1}>0$ by Lemma \ref{L4.31}. Without loss of generality, the possible distributions of the nonnegative integers $p_i$ reduce to $(2,1,0,0,0)$ and $(1,1,1,0,0)$.

Subcase 4.1. $p_0=2$, $p_1=1$, and $p_2=p_3=p_4=0$.

Since $p_2=0$ and $q_{i-1}+q_i=n_ip_i$, we obtain $q_1+q_2=0$, which forces $q_1=q_2=0$. Similarly, $p_3=p_4=0$ gives $q_3=q_4=0$. Now apply $q_{i-1}+q_i=n_ip_i$ with $i=0$ and $i=1$. For $i=0$, we get $q_4+q_0=n_0p_0=2n_0$, and since $q_4=0$, it follows that $q_0=2n_0$. For $i=1$, we get $q_0+q_1=n_1p_1=n_1$, and since $q_1=0$, this gives $q_0=n_1$. Hence, $n_1=2n_0$. On the other hand, recall that $p_i=n_{i-1}+n_{i+1}-\Delta$. From $p_2=0$ we obtain $n_1+n_3=\Delta$, while $p_4=0$ yields $n_3+n_0=\Delta$. Since $n_1+n_3=\Delta$ and $n_3+n_0=\Delta$, we obtain $n_1=n_0$, contradicting the fact that $n_1=2n_0$ and $n_0>0$.

Subcase 4.2. $p_0=p_1=p_2=1$ and $p_3=p_4=0$. 

From $p_3=0$ and $q_{i-1}+q_i=n_ip_i$, we get $q_2+q_3=0$, hence $q_2=q_3=0$. Similarly, $p_4=0$ yields $q_3+q_4=0$, so $q_4=0$. Now apply $q_{i-1}+q_i=n_ip_i$ successively. For $i=0$, we get $q_4+q_0=n_0$, and since $q_4=0$, it follows that $q_0=n_0$. For $i=2$, we get $q_1+q_2=n_2$, and since $q_2=0$, this gives $q_1=n_2$. For $i=1$, we get $q_0+q_1=n_1$. Substituting $q_0=n_0$ and $q_1=n_2$, we obtain $n_1=n_0+n_2$. On the other hand, recall that $p_i=n_{i-1}+n_{i+1}-\Delta$. Since $p_1=1$, we have $n_0+n_2=\Delta+1$. Combining this with $n_1=n_0+n_2$, we get $n_1=\Delta+1$. Finally, since $p_0=1$, we have $n_1+n_4=\Delta+1$. Using $n_1=\Delta+1$, we get $n_4=0$, contradicting the fact $n_4>0$. This completes the proof.
\end{proof}

We next turn to the proof of Theorem \ref{T1.4}.

\begin{proof}
If $n > 24$, the conclusion follows from Theorem \ref{T1.3}. Hence, assume that $n \in [10,24]$. Since $G$ is a regular Class 1 graph, $n$ is even. If $G$ is bipartite, the conclusion follows from Lemma \ref{L2.4}. We may therefore assume that $G$ is non-bipartite. By Lemma \ref{L2.3}, $\Delta \leq \frac{2n}{5}$. Together with $\Delta > \frac{2n-3}{5}$, this implies that $t:=2n-5\Delta \in \{0,1,2\}$. The parity of $n$ and the integrality of $\Delta$ show that the possible triples $(n,\Delta,t)$ are $(10,4,0)$, $(16,6,2)$, $(18,7,1)$, and $(20,8,0)$. If $(n,\Delta,t)=(16,6,2)$, the conclusion follows from Theorem \ref{T1.2}. In the remaining cases, $\Delta>\frac{3n}{8}$, so $G$ is homomorphic to $C_5$ by Lemma \ref{L2.7}. We use the notation of Lemma \ref{L4.31}. If $t=1$, exactly one of the $p_i$ with $i \in [0,4]$ is positive, contradicting Lemma \ref{L4.31}. It remains to consider $t=0$. Since $\Delta>\frac{2n-3}{5}>\frac{n-1}{3}$, Theorem \ref{T1.1} settles every non-extreme vertex-splitting. We may therefore assume that the vertex-splitting is extreme. Equivalently, $G^*$ is obtained from $G$ by replacing $bc$ with the path $(b,ba,a,ac,c)$. Define $B$, $C$, and $R$ as in Lemma \ref{L4.1}. Since $\sum_{i=0}^{4}p_i=0$, Lemma \ref{L4.31} gives $p_i=q_i=0$ for $i \in [0,4]$. Thus, all consecutive pairs of parts are complete bipartite. After relabeling, assume that $b \in V_0$ and $c \in V_1$. Then $R=V_3$, and therefore $G[R]$ has no edges. Since $a,b,c \notin R$ and $G^*$ differs from $G$ only by replacing $bc$ with the path $(b,ba,a,ac,c)$, we have $G^*[R]=G[R]$, and hence $E(G^*[R])=\emptyset$. If $G^*$ contained an edge that is not critical, Lemma \ref{L4.1} would force both endvertices of this edge to lie in $R$, contradicting the fact $E(G^*[R])=\emptyset$. This completes the proof.
\end{proof}

\noindent {\bf Acknowledgements}

This work was supported by Hebei Natural Science Foundation (A2026208005) and Foundation of China Scholarship Council (202508130088).


\begin{thebibliography}{99}

\bibitem{AES1974}
B. Andrásfai, P. Erdős, V.T. Sós, {\it On the connection between chromatic number, maximal clique and minimal degree of a graph}, Discrete Math. {\bf 8} (1974), 205--218.

\bibitem{BR2012}
M.A. Bahmanian, C.A. Rodger, {\it Multiply balanced edge colorings of multigraphs}, J. Graph Theory {\bf 70} (3) (2012), 297--317.


\bibitem{BBG2005}
D. Bokal, G. Brinkmann, S. Grünewald,
{\it Chromatic-index-critical graphs of orders $13$ and $14$},
Discrete Math. {\bf 300} (2005), 16–29.

\bibitem{BV2020}
S. Bonvicini, A. Vietri, {\it A Möbius-type gluing technique for obtaining edge-critical graphs}, Ars Math. Contemp. {\bf 19} (2) (2020), 209--229.

\bibitem{Brandt1997}
S. Brandt, {\it On the structure of dense triangle-free graphs}, Combin. Probab. Comput. {\bf 8} (3) (1999), 237--245.	

\bibitem{BS1998}
G. Brinkmann, E. Steffen, {\it Chromatic-index-critical graphs of orders $11$ and $12$}, Europ. J. Combin. {\bf 19} (8) (1998), 889--900.

\bibitem{CCS2022}
Y. Cao, G. Chen, S. Shan, {\it An improvement to the Hilton-Zhao vertex-splitting conjecture}, Discrete Math. {\bf 345} (2022), 112902.

\bibitem{CY1983}
A.G. Chetwynd, H.P. Yap, {\it Chromatic index critical graphs of order 9},   Discrete Math. {\bf 47} (1983), 23--33.

\bibitem{FW1977}
S. Fiorini, R.J. Wilson, {\it Edge-Colourings of Graphs}, Research Notes in Maths., Pitman, London, 1977.

\bibitem{Gupta1967}
R.P. Gupta, {\it Studies in the Theory of Graphs}, PhD thesis, Tata Institute of Fundamental Research, Bombay, 1967.

\bibitem{Häggkvist1982}
R. Häggkvist, {\it Odd cycles of specified length in non-bipartite graphs}, in: B. Bollobás (Ed.), Graph Theory: Proceedings of the Conference on Graph Theory, Cambridge, North-Holland Math. Stud., vol. 62, North-Holland, Amsterdam, 1982, pp. 89–99.

\bibitem{HZ1997}
A.J.W. Hilton, C. Zhao, {\it Vertex-splitting and chromatic index critical graphs}, Discrete Appl. Math. {\bf 76} (1997), 205--211.

\bibitem{Hoyler1981}
I. Holyer, {\it The NP-completeness of edge-coloring}, SIAM J. Comput. {\bf 10} (4) (1981), 718--720.

\bibitem{Mantel1907}
W. Mantel, Problem 28, Solution by H. Gouwentak, W. Mantel, J. Teixeira de Mattes, F. Schuh and W.A. Wythoff, Wisk. Opgaven {\bf 10} (1907), 60--61.

\bibitem{QF2026}
X. Qi, Y. Feng, {\it On the Hilton-Zhao vertex-splitting conjecture}, Discrete Math. {\bf 349} (2026), 115341.

\bibitem{Song2002}
Z.-X. Song, {\it A further extension of Yap’s construction for $\Delta$-critical graphs}, Discrete Math.
{\bf 243} (2002), 283--290.

\bibitem{SSTF2012}
M. Stiebitz, D. Scheide, B. Toft, L.M. Favrholdt,
{\it Graph Edge Coloring: Vizing’s Theorem and Goldberg’s Conjecture}, Wiley Series in Discrete Mathematics
and Optimization, John Wiley \& Sons, Inc., Hoboken, NJ, 2012.

\bibitem{Vietri2015}
A. Vietri, {\it An analogy between edge colourings and differentiable manifolds, with a new perspective on 3-critical graphs}, Graphs Combin. {\bf 31} (6) (2015), 2425--2435.

\bibitem{Vizing1965-1}
V.G. Vizing, {\it Critical graphs with given chromatic class}, Diskret. Analiz. {\bf 5} (1965), 9--17.

\bibitem{Yap1981}
H.P. Yap, {\it A construction of chromatic index critical graphs}, J. Graph Theory {\bf 5} (1981), 159--163.
\end{thebibliography}
\end{document}